\documentclass{amsart}

\usepackage{mathrsfs}
\usepackage{amsfonts}
\usepackage{amsmath}
\usepackage{amssymb}
\usepackage{mathtools}
\usepackage{float}
\usepackage{xcolor}
\usepackage[inline]{enumitem}
\usepackage{csquotes}
\usepackage{subcaption}
\usepackage{tikz}
\usepackage{tikz-cd}
\usetikzlibrary{arrows,automata}
\usepackage{hyperref}
\usepackage{centernot}
\usepackage{varwidth}
\usepackage{colortbl}
\usepackage{array, multirow, bigdelim, makecell, booktabs} 

\newcommand{\yes}{\cellcolor{green!25}}
\newcommand{\no}{\cellcolor{red!25}}
\newcommand{\dunno}{\cellcolor{blue!25}?}

\newtheorem{theorem}{Theorem}[section]
\newtheorem*{theorem*}{Theorem}
\newtheorem{lemma}[theorem]{Lemma}

\newtheorem{prop}[theorem]{Proposition} 	
\newtheorem{corollary}[theorem]{Corollary}
\newtheorem{conjecture}[theorem]{Conjecture}

\newtheorem{theoremB}{Theorem}

\theoremstyle{definition}

\theoremstyle{remark}
\newtheorem{remark}[theorem]{Remark}

\numberwithin{equation}{section}

\newcommand{\pres}[3]{\textnormal{#1} \langle #2 \mid #3 \rangle}

\newcommand{\xra}[1]{\xrightarrow{\ast}_{#1}}

\newcommand{\trev}{\text{rev}}

\DeclareMathOperator{\UT}{UT}
\DeclareMathOperator{\MT}{M}

\newcommand{\UTZ}[1]{\UT_#1(\overline{\mathbb{Z}})}
\newcommand{\MTZ}[1]{\MT_#1(\overline{\mathbb{Z}})}
\newcommand{\TZ}{\overline{\mathbb{Z}}}
\newcommand{\TR}{\overline{\mathbb{R}}}

\begin{document}

\title{Tropical One-relation Monoids}

\author{Carl-Fredrik Nyberg-Brodda}
\address{Alan Turing Building, Department of Mathematics, University of Manchester, UK}
\email{carl-fredrik.nybergbrodda@manchester.ac.uk}
\thanks{The work in this article was carried out while the author was Dame Kathleen Ollerenshaw Trust Research Associate at the University of Manchester, and the author wishes to thank the Trust for funding this research. The author is currently working as Attach\'e Temporaire d'Enseignement et de Recherche (ATER) at LIGM, Universit\'e Gustave Eiffel (Paris).}

\subjclass[2020]{}

\date{\today}



\begin{abstract}
In 2010, Izhakian \& Margolis proved that the bicyclic monoid admits a representation as a semigroup of upper triangular tropical matrices. We extend this result by classifying all one-relation monoids which admit such representations. We also classify, with two exceptions, all one-relation monoids which admit representations as semigroups of tropical matrices. 
\end{abstract}

\maketitle



\noindent This present article straddles an interesting boundary between combinatorial semigroup theory and tropical mathematics which has been exposed in recent years. To facilitate reading for readers who may be firmly located in one, but not the other, area of research, we will begin by giving an overview of the main ideas used from both areas. We begin by describing identities in combinatorial semigroup theory; after this, we describe tropical mathematics and representations. 

Studying identities satisfied by various classes of semigroups has a long and rich history. Let $X$ be a fixed countably infinite set of \textit{variables}. Let $X^+$ denote the free semigroup on $X$. Equality in free semigroups is always here denoted $\equiv$. A (semigroup) \textit{identity} is an expression of the form $w \bumpeq w'$ for some $w, w' \in X^+$. For any semigroup $S$, we say that $S$ \textit{satisfies} $w \bumpeq w'$ or that the identity $w \bumpeq w'$ \textit{holds in} $S$, sometimes denoted $S \models (w \bumpeq w')$, if for every homomorphism $\varphi \colon X^+ \to S$ we have $\varphi(w) = \varphi(w')$. By the universal property of free semigroups, every such homomorphism is equivalent to a \textit{substitution} of elements of $S$ for the variables occurring in $w$ and $w'$. Thus, the identity $w \bumpeq w'$ holds in $S$ if and only if for every substitution of elements of $S$ for the variables in $w$ and $w'$, the resulting equality holds in $S$. As examples of identities, a semigroup $S$ is commutative if and only if the identity $xy \bumpeq yx$ holds in $S$, and a semigroup is periodic if and only if $x^n \bumpeq x^m$ holds for some $m > n > 0$. Every free semigroup on two or more generators does not satisfy any identity $w \bumpeq w'$ with $w \not\equiv w'$. Identities of the form $w \bumpeq w$, i.e. the identities which hold in every semigroup, are called \textit{trivial}. Monoid and group identities are defined entirely analogously. Clearly, if $S$ satisfies an identity $w \bumpeq w'$ then for any subsemigroup $S' \leq S$, we have that $S'$ satisfies $w \bumpeq w'$. For more details on identities, we refer the reader to the survey \cite{Shevrin1985}.

An early results on identities in semigroups is that by Adian \cite{Adian1962} (detailed proofs were given in \cite[Chapter~IV]{Adian1966}) that the bicyclic monoid satisfies
\begin{equation}\label{Eq:Adian-identity}
xy^2x(xy)xy^2x \bumpeq xy^2x(yx)xy^2x, \tag{AI}
\end{equation}
now frequently referred to as ``Adian's identity''. Adian also claims that \eqref{Eq:Adian-identity} is the shortest identity satisfied by the bicyclic monoid, but, contrary to popular belief, does not actually give a proof of this (correct) fact. The identities satisfied by the bicyclic monoid $\mathcal{B}$ are interesting, and we will expand on them in the tropical context below; before this, we mention briefly that Shneerson \cite{Shneerson1989} proved that there is no \textit{finite basis} for the set of identities satisfied by $\mathcal{B}$, and it follows from Scheiblich \cite{Scheiblich1971} that the monogenic free inverse monoid $\operatorname{FIM}(1)$ satisfies exactly the same identities as $\mathcal{B}$. Adian \cite{Adian1966} also proved that a (not necessarily finitely presented) \textit{special} monoid (i.e. a monoid in which all defining relations are of the form $w=1$) satisfies a non-trivial identity only if it is either monogenic, the bicyclic monoid, or a group. This is one of the early remarkable results of combinatorial semigroup theory. For more on the history of results about special monoids, as well as the author's language-theoretic investigations on the same, see \cite{NybergBrodda2022a}. 

Beyond the bicyclic monoid, much is known about the identities satisfied by one-relation monoids. While it remains a famous open problem whether the word problem is decidable in all one-relation monoids (for an overview of this problem, see the recent survey \cite{NybergBrodda2021a}), it is still possible to make non-trivial statements about the class of all one-relation monoids. In particular, Shneerson \cite{Shneerson1972a, Shneerson1972b} gave a complete classification of the one-relation monoids which satisfy some non-trivial identity. Furthermore, he proved that there exists an algorithm which takes as input two words $w, w' \in X^+$ and a one-relation monoid $M = \pres{Mon}{A}{u=v}$, and which outputs \textsc{yes} if $M$ satisfies $w \bumpeq w'$, and \textsc{no} if it does not.\footnote{This result, for the restricted case of the bicyclic monoid, was later rediscovered by Pastijn \cite[Theorem~4.1]{Pastijn2006}, who seems to have been unaware of Shneerson's much more general result. Pastijn's algorithm uses the same idea as Shneerson's algorithm, which can in rough terms be described as using Minkowski's theorem to find lattice points inside convex polytopes in $\mathbb{R}^n$.
} 

Note that identities in semigroups are fairly closely linked with the existence of free subsemigroups. In the sequel, ``free subsemigroup'' will always refer to one which is free of rank $2$ or greater. Now, if a semigroup $S$ contains a free subsemigroup, then $S$ cannot satisfy any non-trivial identity. It is natural to ask about the converse: if $S$ contains no free subsemigroup, does it follow that $S$ satisfies a non-trivial identity? The answer, it turns out, is no -- Shneerson \cite{Shneerson1993} constructed a semigroup $\Pi$ with polynomial growth (and hence without free subsemigroups) but which satisfies no non-trivial identity. However, in Shneerson's 1972 papers indicated above, Shneerson managed to prove that for the class of one-relation monoids, the equivalence does hold:

\begin{theorem*}[Shneerson, 1972]
A one-relation monoid satisifes a non-trivial identity if and only if it does not contain a free subsemigroup of rank $2$ (if and only if it it has polynomial growth). 
\end{theorem*}

We refer the reader to Shneerson \cite{Shneerson2015} for further details, and note in passing that this also contains a study of identities in one-relation \textit{inverse} monoids. Such monoids appear significantly more complicated than their one-relation ``ordinary'' monoid counterparts (cf. e.g. \cite{Gray2020}). This completes our overview of identities.

We now turn to tropical mathematics. Much like semigroups, this is an area whose axioms are very easy to state. Tropical mathematics is the study of \textit{tropical semirings}. The \textit{integral tropical semiring} $(\TZ, \max, +)$ is the semiring with underlying set $\TZ = \mathbb{Z} \cup \{ -\infty\}$ and operations $\max$ and $+$, defined as usual for elements of $\mathbb{Z}$, and with $\max(-\infty, x) = x$ and $-\infty + x = -\infty$ for all $x \in \TZ$.  For clarity, its operations will sometimes written as $\oplus$ and $\otimes$, rather than $\max$ and $+$. The semiring $\TR$ (with $\TR = \mathbb{R} \cup \{ -\infty\}$) is often called simply the \textit{tropical semiring} $\mathbb{T}$. We will \textit{not}, however, consider $\TR$ in the present article. Note that in the literature, tropical semirings are sometimes called \textit{max-plus algebras}. The \textit{tropical matrix semiring} $\MTZ{n}$ is the semiring of all $(n\times n)$-matrices with entries in $\TZ$, with operations induced as usual from $\TZ$, i.e. $(A \oplus B)_{ij} = A_{ij} \oplus B_{ij}$, and $(A \otimes B)_{ij} = \sum_k A_{ik} \otimes B_{kj}$. The semiring of \textit{upper triangular matrices} $\UTZ{n}$ consists of all $(n \times n)$ upper triangular matrices in $\MTZ{n}$, i.e. those with $M_{ij} = -\infty$ for $i>j$ (in tropical mathematics, $-\infty$ plays the role of $0$ in ``ordinary'' mathematics), and with the same operations as for $\MTZ{n}$. This completes our overview of tropical mathematics. The interested reader may consult \cite{Speyer2009} for a great deal more detail.

The starting point for the present article is a unification of the two areas outlined above. This was initialised by the observation of Margolis \& Izhakian \cite{Izhakian2010} that the bicyclic monoid can be embedded in $\UTZ{2}$. The authors additionally prove that $\UTZ{2}$ satisfies Adian's Identity \eqref{Eq:Adian-identity}, and as a corollary Adian's original result for the bicyclic monoid can be recovered.\footnote{However, the occasional claim to the effect that this provides a ``much simpler'' proof of the fact that the bicyclic monoid satisifes Adian's identity should be regarded with some scepticism: Adian's proof \cite[pp.120--121]{Adian1966} is elementary and just over a page long, quite unlike the proof in \cite{Izhakian2010}. One should also note that the two proofs, when restricted to the case of the bicyclic monoid, are not fundamentally different.} This was later extended in several ways. Izhakian \cite{Izhakian2014} (with erratum  in \cite{Izhakian2016}) and Okni\'{n}ski \cite{Okninski2015} gave a short proof that $\UTZ{n}$ satisfies a non-trivial identity for all $n \geq 2$. An elementary proof of the same fact was later given by Cain et al. \cite{Cain2017}. Later, Daviaud, Johnson~\&~Kambites \cite{Daviaud2018} proved that $\UTZ{2}$ and the bicyclic monoid satisfy the \textit{same} set of identities, thus providing a strong link between classical results on identities in certain classes of semigroups and modern results on identities in tropical matrix semigroups. Especially in the past 10 years, this link has developed greatly, see e.g. \cite{Hollings2012, Chen2016, Izhakian2019, Johnson2019b, Johnson2019, Han2021, Aird2022, Kambites2022}. For example, one of the remarkable results in this area is that by Johnson \& Kambites \cite{Johnson2021} that every plactic monoid satisfies some non-trivial identity.

Let $M$ be a monoid. We say that $M$ is \textit{$\UT$-tropical} if there exists some $n \geq 1$ such that $M$ embeds (as a semigroup) in $\UTZ{n}$. Analogously, we say that $M$ is $\MT$\textit{-tropical} if it embeds in some $\MTZ{n}$. This terminology is currently not in standard use in the existing literature, but it will be very convenient for us. Of course, the property of being a $\UT$-tropical (resp. $\MT$-tropical) monoid is a hereditary property, and hence it follows from Markov's theorem \cite{Markov1951b, Markov1951a} that the satisfiability problem for $\UT$-tropicality (resp. $\MT$-tropicality) is undecidable.\footnote{The author has recently translated all original Russian papers surrounding the Adian--Rabin theorem into English, including Markov's original papers on the semigroup analogue of the theorem (but preceding Adian's papers by half a decade). This can be found online at \cite{NybergBrodda2022b}. } That is:

\begin{theorem*}
There does not exist an algorithm which decides whether a given finitely presented monoid is $\UT$-tropical (resp. $\MT$-tropical) or not.
\end{theorem*}

The present article will show that for \textit{one-relation} monoids, the situation is not quite as dreary. In particular, we will extend Margolis \& Izhakian's result for the bicyclic monoid, by providing tropical representations for other classes of one-relation monoids. To do this, we will rely heavily on Shneerson's classification of the one-relation monoids satisfying some non-trivial identity (see Theorem~\ref{Thm:Shneerson's}). In \S\ref{Sec:Upper-tropical}, we will give a complete classification of the $\UT$-tropical one-relation monoids, as follows:

\begin{theoremB}
Let $M$ be a one-relation monoid. Then $M$ is $\operatorname{UT}$-tropical if and only if $M$ is free monogenic, aperiodic monogenic, the bicyclic monoid, or one of $M_1, M_5, M_6^{(k)}, M_8$, or $M_9^{(k)}$ for some $k \geq 1$.
\end{theoremB}

The names of the monoids are as they are presented in Shneerson's classification (Theorem~\ref{Thm:Shneerson's}). Alternatively, see the left-most column of Table~\ref{Table: Group properties} for the defining relations of these monoids. As a corollary, we conclude: \textit{there exists an algorithm for deciding whether a given one-relation monoid is $\UT$-tropical}. This is in contrast to the above undecidability result. Following this, in \S\ref{Sec:Ordinary tropical}, we will classify -- with two exceptions -- all $\operatorname{M}$\textit{-tropical} one-relation monoids, i.e. those which embed in $\MTZ{n}$ for some $n \geq 1$. This classification is given as follows:

\begin{theoremB}
Let $M$ be a one-relation monoid which is not $M_4$ or $M_7$. Then $M$ is $\operatorname{M}$-tropical if and only if $M$ is monogenic, $\mathcal{K}$, the bicyclic monoid, or one of $M_2, M_3, M_5, M_6^{(k)}, M_8$, or $M_9^{(k)}$, for some $k \geq 1$.
\end{theoremB}

The two exceptions, $M_4$ and $M_7$, are anti-isomorphic, and hence either both embed in $\MTZ{n}$ for some $n \geq 1$, or else neither of them do. We treat these open cases briefly in \S\ref{Subsec:Open_m4_m7}. Our results are summarised in Table~\ref{Table: Group properties}. Finally, we pose the following question which falls naturally out of our work. Let $M$ a one-relation monoid satisfying a non-trivial identity. Is there some $n \geq 1$ and a variety $\operatorname{S}_n(\overline{\mathbb{Z}})$ of tropical matrices satisfying precisely the same identities as $M$? A positive answer in the case of the bicyclic monoid, with $\operatorname{S}_n(\overline{\mathbb{Z}}) = \UTZ{2}$, was obtained, as mentioned earlier, by Daviaud, Johnson \& Kambites \cite{Daviaud2018}. 

\

\begin{table}[h]
\centering
\scalebox{0.85}{
\begin{tabular}{c|c|c|||c|c|}
\cline{2-5}  & $\UT$-tropical? & $\UT$-rank & $\operatorname{M}$-tropical? & $\operatorname{M}$-rank\\ \hline 
\multicolumn{1}{|l|}{$\mathbb{N}$ \qquad\hfill $(\varnothing)$}  & \checkmark \yes \quad  &  1 & \checkmark \yes \quad & 1\\\hline
\multicolumn{1}{|l|}{$C_{k,\ell}$ \quad\hfill $(a^k = a^\ell)$}  & \no $\times$ \quad (Lem.~\ref{Lem:C-k-ell-not-UT-tropical}) &  -- & \checkmark \yes \quad  & $\leqslant k$ \\\hline
\multicolumn{1}{|l|}{$C_{\ell+1,\ell}$ \quad\hfill $(a^{\ell+1} = a^{\ell})$}  & \yes \checkmark \quad (Lem.~\ref{Lem:Aperiodic-monogenic-are-UT}) &  $\leqslant \ell$ & \checkmark \yes (Lem.~\ref{Lem:Aperiodic-monogenic-are-UT}) \quad & $\leqslant \ell$ \\\hline
\multicolumn{1}{|l|}{$\mathcal{B}$ \quad\hfill $(ab=1)$}  & \checkmark \yes \quad \cite[Prop.~4.3]{Izhakian2010}  &  2 & \checkmark \yes \quad \cite[Prop.~4.3]{Izhakian2010} & 2\\\hline
\multicolumn{1}{|l|}{$\mathcal{K}$ \qquad\hfill $(abba=1)$}  & \no $\mathbf{\times}$ \quad (Cor.~\ref{Cor:Canc.-subsems-of-UTZ-are-comm})  & -- & \checkmark \yes \quad (Lem.~\ref{Lem:Klein-bottle-embeds-in-U})  & 2 \\\hline
\multicolumn{1}{|l|}{$M_1$ \quad\hfill $(ab=ba)$}  & \checkmark\yes \quad  & 2 & \checkmark\yes \quad  & 2\\ \hline
\multicolumn{1}{|l|}{$M_2$ \quad \hfill $(a^2 = b^2)$}  & $\times$\no \quad (Lem.~\ref{Lem:M2-M3-not-UT}) & -- & \checkmark\yes \quad (Lem.~\ref{Lem:a2=b2-is_tropical_real}) & 2 \\ \hline
\multicolumn{1}{|l|}{$M_3$ \quad \hfill $(aba=b)$}  & $\times$\no \quad (Lem.~\ref{Lem:M2-M3-not-UT}) & -- & \checkmark\yes \quad (Lem.~\ref{Lem:aba=b-is_tropical_real})  & 2 \\ \hline
\multicolumn{1}{|l|}{$M_4$ \qquad \hfill $(aba^2=ba)$}  & $\times$\no \quad (Lem.~\ref{Lem:M3-in-M4-and-M7}) & -- & \dunno  &  \\ \hline
\multicolumn{1}{|l|}{$M_5$ \qquad\hfill $(aba=ba)$}  & \yes \checkmark \quad (Lem.~\ref{Lem:aba=ba(M5)}) & $\leqslant 4$ & \yes \checkmark \quad (Lem.~\ref{Lem:aba=ba(M5)}) & $\leqslant 4$ \\ \hline
\multicolumn{1}{|l|}{$M_6^{(k)}$ \quad\hfill $(ab=b^k)$} & \yes \checkmark \quad (Lem.~\ref{Lem:ab=bk_when_k=1}+\ref{Lem:ab=bk_bigger_k}) & 2 & \yes \checkmark \quad (Lem.~\ref{Lem:ab=bk_when_k=1}+\ref{Lem:ab=bk_bigger_k}) & 2 \\ \hline
\multicolumn{1}{|l|}{$M_7$ \qquad\hfill $(a^2ba=ab)$}  & $\times$\no \quad (Lem.~\ref{Lem:M3-in-M4-and-M7}) & -- & \dunno &  \\ \hline
\multicolumn{1}{|l|}{$M_8$ \qquad\hfill $(aba=ab)$}  & \yes \checkmark \quad (Lem.~\ref{Lem:aba=ab(M8)}) & $\leqslant 4$ & \yes \checkmark \quad (Lem.~\ref{Lem:aba=ab(M8)}) & $\leqslant 4$ \\ \hline
\multicolumn{1}{|l|}{$M_9^{(k)}$ \quad\hfill $(ba=b^k)$} & \yes \checkmark \quad (Lem.~\ref{Lem:ba=bk_when_k=1}+\ref{Lem:ba=bk_bigger_k}) & 2 & \yes \checkmark \quad (Lem.~\ref{Lem:ba=bk_when_k=1}+\ref{Lem:ba=bk_bigger_k}) & 2  \\ \hline
\end{tabular}
}
\vspace{0.5cm}
\caption{The tropical one-relation monoids. A name and a defining relation for each monoid is given in the leftmost column. Every $\leqslant$-symbol and every question mark (inside a blue box) indicates an open problem to be solved. Here, $\UT$-rank (resp. $\operatorname{M}$-rank) indicates the least $n$ such that the monoid embeds as a multiplicative subsemigroup of $\UTZ{n}$ (resp. $\MTZ{n}$). }
\label{Table: Group properties}
\end{table}

\setcounter{theoremB}{0}

\clearpage
\section{Preliminaries}

\subsection{Identities in one-relation monoids}

Even a small amount of work with one-relation monoids reveals that they often contain non-commutative free subsemigroups. It therefore seems natural to wish to classify precisely which one-relation monoids $M$ do not contain free subsemigroups, and to classify which such $M$ satisfy some non-trivial identity. Both these questions were answered by Shneerson in 1972 -- we give an overview of this result now (Theorem~\ref{Thm:Shneerson's}). The interested reader may note that Shneerson \cite{Shneerson1989} has studied the axiomatic ranks of the varities generated by the one-relation monoids satisfying a non-trivial identity, a subject as closely related to identities as it is outside of the scope of the present article. 

First, we remark that it is obvious that any monogenic one-relation monoid satisfies a non-trivial identity, being commutative. We continue with the special one-relation monoids. As mentioned in the introduction, Adian \cite{Adian1966} proved that \textit{any} special monoid satisfying a non-trivial identity is either monogenic, the bicyclic monoid, or a group. Hence, any special one-relation monoid, which is not the bicyclic monoid or $\mathbb{N}$, satisfies a non-trivial identity if and only if it is a positive one-relator group satisfying a non-trivial identity (here positive is used in the sense of Baumslag  \cite{Baumslag1971}, see also Perrin \& Schupp \cite{Perrin1984}). Which positive (non-monogenic) one-relator groups, then, satisfy some non-trivial identity? Shneerson \cite{Shneerson1972a}, relying on old results by Magnus \cite{Magnus1930}, proved that there is only one such group: namely, the one-relator group $\mathcal{K} = \pres{Mon}{a,b}{abba=1}$. This is indeed a group, as the defining relation tells us that $a$ is right invertible ($a \cdot bba = 1$), and that $a$ is left invertible ($abb \cdot a = 1$). Thus $a$ is invertible in $M$, so we conclude that also $bbaa = 1$ and $aabb = 1$. Thus $b$ is invertible, so $\mathcal{K}$ is a group. Hence 
\[
\mathcal{K} = \pres{Mon}{a,b}{abba=1} = \pres{Gp}{a,b}{abba=1} \cong \pres{Gp}{a,b}{bab^{-1} = a^{-1}} \cong \mathbb{Z} \rtimes \mathbb{Z},
\]
the fundamental group of the Klein bottle. This is a virtually abelian group, having $\mathbb{Z}^2$ as an index $2$ (normal) subgroup, as can be seen by passing to a degree $2$ cover of the Klein bottle by the $2$-torus. Hence, $\mathcal{K}$ satisfies the identity $x^2 y^2 \bumpeq y^2 x^2$. 

We can now state Shneerson's classification of the one-relation monoids satisfying a non-trivial identity, with the bulk coming from the non-special case:

\begin{theorem}[Shneerson \cite{Shneerson1972a, Shneerson1972b}\footnote{These two articles are currently not readily accessible. The author of the present article has translated Shneerson's articles into English, having received copies of the Russian originals from L. M. Shneerson and M. V. Volkov, for which the author expresses his gratitude. These English translations will soon be available online. }]\label{Thm:Shneerson's}
Let $M = \pres{Mon}{A}{u=v}$. Then $M$ satisfies a non-trivial identity if and only if $M$ is monogenic, $\mathcal{K}$, the bicyclic monoid, or $|A|=2$, say $A = \{ a, b \}$, and $u=v$ is one of the following: \\
\[
\begin{tabular}{lll}
\textnormal{(1)} $ab=ba$, & \textnormal{(2)} $a^2=b^2$, & \textnormal{(3)} $aba = b$, \\
\textnormal{(4)} $aba^2 = ba$, & \textnormal{(5)} $aba = ba$, & \textnormal{(6)} $ab = b^k, \quad (k = 1, 2, \dots)$ \\
\textnormal{(7)} $a^2ba=ab$, & \textnormal{(8)} $aba = ab$, & \textnormal{(9)} $ba = b^k, \quad (k = 1, 2, \dots).$
\end{tabular}
\]
These monoids satisfy, respectively, the following identities:
\[
\begin{tabular}{lll}
\textnormal{(1)} $xy \bumpeq yx$, & \textnormal{(2)} $x^2y^2 \bumpeq y^2x^2$, & \textnormal{(3)} $x^2y^2 \bumpeq y^2x^2$, \\
\textnormal{(4)} $(xyx)^2(yx)^2 \bumpeq (yx)^2(xyx)^2$, & 
\textnormal{(5)} $xyxyx \bumpeq yx^2yx$, & 
\textnormal{(6)} $xyxyx \bumpeq yx^2yx$, \\
\textnormal{(7)} $(xyx)^2(xy)^2 \bumpeq (xy)^2(xyx)^2$, & 
\textnormal{(8)} $xyxyx \bumpeq xyx^2y$, & 
\textnormal{(9)} $xyxyx \bumpeq xyx^2y$.
\end{tabular}
\]
Furthermore, in each of the above monoids $M$, given any words $w_1, w_2 \in X^\ast$ one can effectively decide whether $w_1 \bumpeq w_2$ holds in $M$.
\end{theorem}

Throughout this article, we shall refer to the monoids in Theorem~\ref{Thm:Shneerson's} as they are numbered there. For example, the monoid $\pres{Mon}{a,b}{a^2ba=ab}$, being case (7), will be referred to as $M_7$. Analogously, the monoid $\pres{Mon}{a,b}{ab=b^k}$, being case (6) with parameter $k \geq 1$, will be called $M^{(k)}_6$. 

\begin{remark}
We remark that Theorem~\ref{Thm:Shneerson's} is also a complete classification of the one-relation presentations of the (non-special) one-relation monoids satisfying some non-trivial identity. That is, given any non-special one-relation monoid $M = \pres{Mon}{A}{u=v}$, one can decide whether $M$ satisfies some non-trivial identity by simply checking if its relation is one of the relations in Theorem~\ref{Thm:Shneerson's}. 
\end{remark}

\begin{remark}
It follows immediately from Adian's theory of left/right cycle-free semigroups from \cite[Chapter~II]{Adian1966} that in Theorem~\ref{Thm:Shneerson's}, the monoids (1)--(3) are cancellative; the monoids (4)--(6) are left (but not right) cancellative; and (7)--(9) are right (but not left) cancellative.
\end{remark}

\subsection{Identities in tropical matrix semigroups}

We give a very quick summary of some useful tools arising from identities in tropical matrix semigroups. Let $X \in \MTZ{n}$. Then $X^\top$ will denote the (usual) transpose of $X$. Furthermore, $\mathbf{0}_{n}$ will denote the $(n \times n)$-matrix in which all entries are $-\infty$. The notation $\mathbf{0}$ arises from the fact that this is the zero matrix of the tropical world, in the sense that 
\begin{align*}
X \oplus \mathbf{0}_{n} &= \mathbf{0}_{n}\oplus X = X \quad \text{and}\\
X \otimes \mathbf{0}_{n} &= \mathbf{0}_{n}\otimes X = \mathbf{0}_{n}, 
\end{align*}
for all $X \in \MTZ{n}$. Analogously, $\mathbf{I}_n$ will denote the $(n \times n)$-matrix in which all entries on the diagonal are $0$, and all other entries are $-\infty$. This is the tropical analogue of the usual identity matrix, in the sense that for all $X \in \MTZ{n}$, we have
\[
X \otimes \mathbf{I}_{n} = \mathbf{I}_{n}\otimes X = X.
\]

We shall save some work in finding tropical representations of our monoids by using transposes. Let $w \equiv a_1 a_2 \cdots a_n$ be a word with $a_i \in A$. Then $w^\trev$ will denote the \textit{reversed} word $a_n a_{n-1} \cdots a_2 a_1$, i.e. the word $w$ read backwards. Let $M = \pres{Mon}{A}{u_i = v_i \: (i \in I)}$. Then $M^\trev$ denotes the \textit{reversed} monoid, i.e. the monoid $M = \pres{Mon}{A}{u_i^\trev = v_i^\trev \: (i \in I)}$. Clearly, $M$ and $M^\trev$ are anti-isomorphic. The following lemma is also easy to prove. 

\begin{lemma}\label{Lem:If-reverse-then-transpose}
Let $M = \pres{Mon}{A}{u_i = v_i \: (i \in I)}$, and suppose $M$ is isomorphic to the subsemigroup of $\MTZ{n}$ generated by the matrices $X_1, X_2, \dots, X_m$. Then $M^{\textnormal{rev}}$ is isomorphic to the semigroup generated by the matrices $X_1^\top, X_2^\top, \dots, X_m^\top$.
\end{lemma}

In particular, it suffices to restrict our attention to the left cycle-free one-relation monoids in our classification. For example, $M_5$ is $\UT$-tropical if and only if $M_8$ is. 

Following \cite{Cain2017}, we define the functions 
\begin{alignat*}{2}
U_0(p,q) &= p \qquad\qquad &&V_0(p,q)=q \\
U_1(p,q) &= pqppq \qquad\qquad &&V_1(p,q)=pqqpq
\end{alignat*}
and for $i>1$ inductively define
\[
U_i(p,q) = U_1(U_{i-1}(p,q), V_{i-1}(p,q)) \quad V_i(p,q) = V_1(U_{i-1}(p,q), V_{i-1}(p,q)).
\]
For example, it is easy to check that $U_1(xy,yx) \bumpeq V_1(xy,yx)$ is simply Adian's identity, i.e. \eqref{Eq:Adian-identity}. This is satisfied by $\UTZ{2}$. More generally, we have:

\begin{prop}[{\cite[Theorem~4.2]{Cain2017}}]\label{Prop:Cain_UT_id}
For every $n \geq 1$, the monoid $\UTZ{n}$ satisfies the identity $U_{n-1}(xy,yx) \bumpeq V_{n-1}(xy,yx)$.
\end{prop}

\begin{corollary}\label{Cor:Canc.-subsems-of-UTZ-are-comm}
Let $M$ be a cancellative subsemigroup of $\UTZ{n}$ for some $n \geq 1$. Then $M$ is commutative.
\end{corollary}
\begin{proof}
It suffices to prove, by Prop.~\ref{Prop:Cain_UT_id}, the stronger claim that if the identity $U_{n}(xy,yx) \bumpeq V_{n}(xy,yx)$ holds in a cancellative semigroup $S$ for some $n \geq 1$, then $xy \bumpeq yx$ holds in $S$. The proof is by induction on $n$. If $n=1$, then there is nothing to show. Suppose the claim holds for $n< n_0$, where $n_0 > 1$. We prove the claim for $n=n_0$. Let $U' = U_{n_0-1}(xy,yx)$ and $V' = V_{n_0-1}(xy,yx)$. As $U_n(xy, yx) = U_{n_0}(U'V', V'U')$ and $V_{n_0}(xy,yx) = V_1(U'V', V'U')$, we have that
\begin{equation}\label{Eq:UV-U'V'-Cain}
U'V'V'U' (U'V')U'V'V'U' \bumpeq U'V'V'U' (V'U') U'V'V'U'
\end{equation}
holds in $S$. As $S$ is cancellative, \eqref{Eq:UV-U'V'-Cain} implies that $U'V' \bumpeq V'U'$ in $S$. Hence, by the inductive hypothesis, $xy \bumpeq yx$ holds in $S$.
\end{proof}


\section{One-relation monoids in $\UTZ{n}$}\label{Sec:Upper-tropical}

\noindent In this section, we will completely describe the $\operatorname{UT}$-tropical one-relation monoids, i.e. those which embed into $\operatorname{UT}_n(\overline{\mathbb{Z}})$ for some $n \geq 1$, by proving the following:

\begin{theoremB}\label{Thm:Main_UT}
Let $M$ be a one-relation monoid. Then $M$ is $\operatorname{UT}$-tropical if and only if $M$ is free monogenic, aperiodic monogenic, the bicyclic monoid, or one of $M_1, M_5, M_6^{(k)}, M_8$, or $M_9^{(k)}$ for some $k \geq 1$.
\end{theoremB}

We begin the proof immediately. We will first show the converse (``positive'') direction, by exhibiting embeddings of the claimed $\UT$-tropical one-relation monoids. 

\subsection{Positive results}

Obviously, $\mathbb{N}$ and $M_1 \cong \mathbb{N}^2$ are both $\UT$-tropical. It remains to show that any aperiodic monogenic one-relation monoid, $M_5, M_6^{(k)}, M_8$ and $M_9^{(k)}$ are all $\operatorname{UT}$-tropical (forming the ``positive'' results). For completeness, we also give a quick, streamlined proof that the bicyclic monoid is $\UT$-tropical.

\
\begin{center}
\textit{Aperiodic monogenic monoids are $\UT$-tropical}
\end{center}

\begin{lemma}\label{Lem:Aperiodic-monogenic-are-UT}
Let $C_{\ell+1,\ell} = \pres{Mon}{a}{a^{\ell+1} = a^\ell}$. Then $C_{\ell+1,\ell}$ is $\UT$-tropical.
\end{lemma}
\begin{proof}
Let $A$ be the $(\ell \times \ell)$-matrix with all entries on and below the diagonal as $-\infty$, and all entries above the diagonal $1$. Then all powers of $A^i$ are distinct for $1 \leq i \leq \ell$, and $A^\ell = \mathbf{-\infty}_{\ell \times \ell} = A^{\ell+1}$. Hence $C_{\ell+1, \ell}$ is isomorphic to the submonoid of $\UTZ{\ell}$ generated by $A$ (note that $A^0 = I_{n \times n} \neq A^{\ell}$). 
\end{proof}

\
\begin{center}
\textit{The bicyclic monoid $\mathcal{B}$ is $\UT$-tropical}
\end{center}

\begin{lemma}[Izhakian \& Margolis]\label{Lem:bicyclic_embeds}
The bicyclic monoid $\mathcal{B}= \pres{Mon}{a,b}{ab=1}$ is isomorphic to the subsemigroup of $\UTZ{2}$ generated by the two matrices
\[
A = \begin{pmatrix}
0 & 1 \\
-\infty & 1 
\end{pmatrix} \qquad \textnormal{and} \qquad B = \begin{pmatrix}
0 & 0 \\
-\infty & -1 
\end{pmatrix}.
\]
\end{lemma}
\begin{proof}
Let $\psi \colon \mathcal{B} \to \UTZ{2}$ be the semigroup homomorphism defined by $a \mapsto A, b \mapsto B$ and $1 \mapsto AB$. It is easy to verify that for $i, j \geq 0$, we have 
\[
B^i A^j = \begin{pmatrix}
0 & j \\
-\infty & j-i 
\end{pmatrix}.
\] Hence the matrix $B^i A^j$ uniquely determines the pair $(i, j)$ of natural numbers; as the language $b^\ast a^\ast$ is a language of unique representatives for the elements of $\mathcal{B}$, it follows that $\psi$ is injective, and so an isomorphism from $\mathcal{B}$ onto $\langle A, B \rangle$. 
\end{proof}

\begin{remark}
In Lemma~\ref{Lem:bicyclic_embeds}, the identity element does not map to $I_2$, but rather to the idempotent in $\UTZ{2}$ with all entries above the diagonal equal to $0$. 
\end{remark}

\
\begin{center}
\textit{$M_6^{(k)}$ and $M_9^{(k)}$ are $\operatorname{UT}$-tropical for all $k \geq 1$}
\end{center}

\noindent We begin with a simple normal form lemma for $M_6^{(k)} = \pres{Mon}{a,b}{ab = b^k}$. By using the finite complete rewriting system with the single rule $ab \to b^k$, it is clear that in $M_6^{(k)}$ every word $w \in \{ a, b\}^\ast$ is equal to some word of the form $b^\alpha a^\beta$, where $\alpha, \beta \geq 0$, and that this representation, being irreducible, is unique. Furthermore, this yields the following rule for multiplication in $M_6^{(k)}$:

\begin{equation}\label{Eq:M6k_normal_form}
b^{\alpha_1}a^{\beta_1} \cdot b^{\alpha_2} a^{\beta_2} =
\begin{cases}
b^{\alpha_1 + \beta_1(k-1) + \alpha_2} a^{\beta_2}, & \text{if } \alpha_2 \neq 0\\
b^{\alpha_1} a^{\beta_1+\beta_2}, & \text{otherwise.}
\end{cases}
\end{equation}
This normal form was used already by \cite{Shneerson1972b}. The fact that the multiplication \eqref{Eq:M6k_normal_form} can be expressed as a linear function of the parameters $\alpha_1, \alpha_2, \beta_1, \beta_2$ is a good indicator that $M_6^{(k)}$ does indeed have a tropical representation; the fact that two parameters $\alpha, \beta$ are used to express the normal form suggests that it has a representation of rank $2$. We will now use the normal form to verify this, i.e. that a certain monoid homomorphism from $M_6^{(k)}$ into $\UTZ{2}$ is injective. The case $k=1$ is distinguished, and we treat this first; the case $k>1$ is uniform in $k$. 

\begin{lemma}\label{Lem:ab=bk_when_k=1}
The monoid $M_6^{(1)}= \pres{Mon}{a,b}{ab=b}$ is isomorphic to the submonoid of  $\UTZ{2}$ generated by the two matrices
\[
A = \begin{pmatrix}
0 & -\infty \\
-\infty & 1 
\end{pmatrix} \qquad \textnormal{and} \qquad B = \begin{pmatrix}
1 & 1 \\
-\infty & -\infty
\end{pmatrix}.
\]
\end{lemma}
\begin{proof}
Let $\psi \colon M_6^{(1)} \to \UTZ{2}$ be the monoid homomorphism defined by $a \mapsto A$ and $b \mapsto B$. Multiplying, one finds $AB=B$, so $\psi$ is indeed a monoid homomorphism. One can quickly verify that for all $\alpha > 0, \beta \geq 0$, we have 
\[
A^\beta = \begin{pmatrix}
0 & -\infty \\
-\infty & \beta 
\end{pmatrix} \quad \textnormal{and} \quad B^\alpha = \begin{pmatrix}
\alpha & \alpha \\
-\infty & -\infty
\end{pmatrix}.
\]
By a straightforward multiplication, we find
\begin{equation}\label{Eq:M6_matrices_mult}
B^{\alpha} A^{\beta} =
 \begin{cases}
\begin{pmatrix}
\alpha & \alpha+\beta \\
-\infty & -\infty
\end{pmatrix}  & \text{if } \alpha \neq 0,\\
\begin{pmatrix}
0 & -\infty \\
-\infty & \beta
\end{pmatrix} & \text{otherwise.}
\end{cases}
\end{equation}
Given any matrix of either the first or second form of the right-hand side of \eqref{Eq:M6_matrices_mult}, one easily determines a unique pair $(\alpha, \beta)$ of natural numbers that gave rise to it, so $\psi$ is injective.
\end{proof}

\begin{lemma}\label{Lem:ab=bk_bigger_k}
The monoid $M_6^{(k)}= \pres{Mon}{a,b}{ab=b^k}$, where $k>1$, is isomorphic to the subsemigroup of  $\UTZ{2}$ generated by the two matrices
\[
A = \begin{pmatrix}
k-1 & -\infty \\
-\infty & 0 
\end{pmatrix} \qquad \textnormal{and} \qquad B = \begin{pmatrix}
1 & 1 \\
-\infty & -\infty
\end{pmatrix}.
\]
\end{lemma}
\begin{proof}
Let $k>1$, and let $\psi \colon M_6^{(k)} \to \UTZ{2}$ be the map defined by $a \mapsto A$ and $b \mapsto B$. One can quickly verify that for all $\alpha > 0, \beta \geq 0$, we have 
\[
A^\beta = \begin{pmatrix}
(k-1)\beta & -\infty \\
-\infty & 0 
\end{pmatrix} \quad \textnormal{and} \quad B^\alpha = \begin{pmatrix}
\alpha & \alpha \\
-\infty & -\infty
\end{pmatrix}.
\]
Multiplying, one finds 
\[
AB=\begin{pmatrix}
k & k \\
-\infty & -\infty 
\end{pmatrix} = B^k
\]
so $\psi$ is indeed a monoid homomorphism. By a straightforward multiplication, we find
\begin{equation}\label{Eq:M6_k_bigger_matrices_mult}
B^{\alpha} A^{\beta} =
 \begin{cases}
\begin{pmatrix}
\alpha + (k-1)\beta & \alpha \\
-\infty & -\infty
\end{pmatrix}  & \text{if } \alpha \neq 0,\\
\begin{pmatrix}
(k-1)\beta & -\infty \\
-\infty & 0
\end{pmatrix}, & \text{otherwise.}
\end{cases}
\end{equation}
Given any matrix of either the first or second form of the right-hand side of \eqref{Eq:M6_k_bigger_matrices_mult}, one easily determines a unique pair $(\alpha, \beta)$ of natural numbers that gave rise to it, so $\psi$ is injective.
\end{proof}

Using the representations in Lemma~\ref{Lem:ab=bk_when_k=1} and Lemma~\ref{Lem:ab=bk_bigger_k}, one can now verify \eqref{Eq:M6k_normal_form} using only tropical matrix multiplication. The dual right cancellative cases of $M_9^{(k)}$ are hence also $\operatorname{UT}$-tropical, by transposing the generators and using Lemma~\ref{Lem:If-reverse-then-transpose}.

\begin{lemma}\label{Lem:ba=bk_when_k=1}
The monoid $M_9^{(1)}= \pres{Mon}{a,b}{ba=b}$ is isomorphic to the submonoid of  $\UTZ{2}$ generated by the two matrices
\[
A = \begin{pmatrix}
1 & -\infty \\
-\infty & 0 
\end{pmatrix} \qquad \textnormal{and} \qquad B = \begin{pmatrix}
-\infty & 1 \\
-\infty & 1 \\
\end{pmatrix}.
\]
\end{lemma}

\begin{lemma}\label{Lem:ba=bk_bigger_k}
The monoid $M_9^{(k)}= \pres{Mon}{a,b}{ba=b^k}$, where $k>1$, is isomorphic to the subsemigroup of  $\UTZ{2}$ generated by the two matrices
\[
A = \begin{pmatrix}
0 & -\infty \\
-\infty & k-1
\end{pmatrix} \qquad \textnormal{and} \qquad B = \begin{pmatrix}
-\infty & 1 \\
-\infty & 1
\end{pmatrix}.
\]
\end{lemma}

This completes the proofs that $M_6^{(k)}$ and $M_9^{(k)}$ are $\operatorname{UT}$-tropical for all $k \geq 1$.

\
\begin{center}
\textit{$M_5$ and $M_8$ are $\operatorname{UT}$-tropical}
\end{center}

\noindent We will require a normal form lemma for $M_5$, which is fairly simple to prove directly, but requires some care. The following is essentially observed by Shneerson \cite{Shneerson1972b} (who denotes $M_5$ as $\Pi^{(3)}$), although his multiplication table omits one case. 

\begin{lemma}[Shneerson]\label{Lem:Shneerson-M5_aba=ba}
Let $M_5 = \pres{Mon}{a,b}{aba=ba}$. Then every word is equal in $M_5$ to some unique word of the form $(ba)^\alpha a^\beta b^\gamma$, with $\alpha, \beta, \gamma \geq 0$, and multiplication in $M_5$ is given by the rule 
\[
(ba)^{\alpha_1}a^{\beta_1} b^{\gamma_1} \cdot (ba)^{\alpha_2} a^{\beta_2} b^{\gamma_2}=
\begin{cases}
(ba)^{\alpha_1 + \gamma_1 + \alpha_2} a^{\beta_2} b^{\gamma_2}, & \text{if } \alpha_2 \neq 0\\
(ba)^{\alpha_1 + \gamma_1} a^{\beta_2-1} b^{\gamma_2}, & \text{if } \alpha_2 = 0, \beta_2 \neq 0, \gamma_1 \neq 0,\\
(ba)^{\alpha_1} a^{\beta_1+\beta_2} b^{\gamma_2}, & \text{if } \alpha_2 = 0, \beta_2 \neq 0, \gamma_1 = 0, \\
(ba)^{\alpha_1} a^{\beta_1} b^{\gamma_1 + \gamma_2}, & \text{if } \alpha_2 = 0, \beta_2 = 0.
\end{cases}
\]
\end{lemma}

We shall now use this normal form to give a tropical representation of $M_5$. 

\begin{lemma}\label{Lem:aba=ba(M5)}
The monoid $M_5 = \pres{Mon}{a,b}{aba=ba}$, is isomorphic to the subsemigroup of  $\UTZ{4}$ generated by the two matrices
\[
A = \begin{pmatrix}
0 & 1 & -\infty & 0 \\
-\infty &  1 & 0 & -1 \\
-\infty & -\infty & -\infty & -\infty \\
-\infty & -\infty & -\infty & -\infty \\
\end{pmatrix} \qquad \textnormal{and} \qquad B = \begin{pmatrix}
1 & 0 & 0 & -1 \\
-\infty & -\infty & -\infty & 0 \\
-\infty & -\infty & 0 & -\infty \\
-\infty & -\infty & -\infty & 1 \\
\end{pmatrix}.
\]
\end{lemma}
\begin{proof}
One may verify that 
\[
ABA = \begin{pmatrix}
-1 & 1 & 0 & -1 \\
-\infty & -\infty & -\infty & 0 \\
-\infty & -\infty & -\infty & -\infty \\
-\infty & -\infty & -\infty & -\infty \\
\end{pmatrix} = BA,
\]
and therefore the map $\psi \colon M_5 \to \UTZ{4}$ defined by $a \mapsto A$ and $b \mapsto B$ extends to a monoid homomorphism. To show that $\psi$ is injective, by Lemma~\ref{Lem:Shneerson-M5_aba=ba}, it suffices to show that the map $(\alpha, \beta, \gamma) \mapsto (BA)^\alpha A^\beta B^\gamma$ is injective. 

If $\alpha > 0, \beta > 1$, and $\gamma > 0$, then a straightforward multiplication shows that
\[
(BA)^\alpha A^\beta B^\gamma = \begin{pmatrix}
-\alpha-\gamma & -\alpha-\gamma+1 & -\alpha+\beta+1 & -\alpha+\beta+\gamma+1 \\
-\infty & -\infty & -\infty & -\infty \\
-\infty & -\infty & -\infty & -\infty \\
-\infty & -\infty & -\infty & -\infty \\
\end{pmatrix}.
\]
The first and second column of this matrix are linearly dependent; on the other hand, taking the first, third, and fourth columns of the matrix, and observing that 
\[
\det \begin{pmatrix}
-1 &0 &-1 \\
-1 &1 &0 \\
-1 &1 &1 \\
\end{pmatrix} = -1 \neq 0,
\]
we see that these columns determine $\alpha, \beta$ and $\gamma$ uniquely. 

In the cases when $\alpha = 0$, $\beta \leq 1$, or $\gamma = 0$, we find that the matrix $X = (BA)^\alpha A^\beta B^\gamma$ similarly determines the triple $(\alpha, \beta, \gamma)$ uniquely, and that furthermore in each of these cases $X$ has non-infinite entries in distinguishing places; for example, when $\alpha = 0$, we have $X_{3,4} = \beta -1$, and if additionally $\gamma = 0$, then $X_{2,4} = \beta$; otherwise, if $\gamma > 0$, then $X_{2,4} = -\infty$. We leave the details to the reader. In particular, we conclude that the map $(\alpha, \beta, \gamma) \mapsto (BA)^\alpha A^\beta B^\gamma$ is injective.
\end{proof}

By transposing, we hence find the following.

\begin{lemma}\label{Lem:aba=ab(M8)}
The monoid $M_8 = \pres{Mon}{a,b}{aba=ab}$, is isomorphic to the submonoid of $\UTZ{4}$ generated by the two matrices
\[
A = \begin{pmatrix}
-\infty & -\infty & -1 & 0 \\
-\infty & -\infty & 0 & -\infty \\
-\infty & -\infty & 1 & 1 \\
-\infty & -\infty & -\infty & 0 \\
\end{pmatrix} \qquad \textnormal{and} \qquad B = \begin{pmatrix}
1 & -\infty & 0 & -1 \\
-\infty & 0 & -\infty & 0 \\
-\infty & -\infty & -\infty & 0 \\
-\infty & -\infty & -\infty & 1 \\
\end{pmatrix}.
\]
\end{lemma}

This completes the proof of the ``positive'' direction of Theorem~\ref{Thm:Main_UT}. 

Note that we have produced an embedding of $M_5$ (and its dual $M_8$) into $\UTZ{4}$. Obviously, for the other $\UT$-tropical embeddings, the dimension of the matrices are the smallest possible. We do not know if this is the case for $M_5$ and $M_8$. We make the following conjecture:

\begin{conjecture}
The monoid $M_5 = \pres{Mon}{a,b}{aba=ba}$ has $\UT$-rank exactly $4$. 
\end{conjecture}

In other words, while we do not know if $M_5$ (resp. $M_8$) can be embedded in $\UTZ{2}$ or $\UTZ{3}$, we conjecture that this is not possible.

\subsection{Negative results}

Let $M$ be a $\UT$-tropical one-relation monoid. As $\UTZ{n}$ satisfies a non-trivial identity for all $n\geq 1$ (see introduction), it follows that $M$ does, too. Hence, by Theorem~\ref{Thm:Shneerson's} either $M$ is monogenic, $\mathcal{K}$, the bicyclic monoid, or $|A| = 2$ (say $A= \{a, b\}$) and $M$ is one of the one-relation monoids $M_1, M_2, \dots, M_9$. We will first show that non-aperiodic finite monogenic one-relation monoids are not $\UT$-tropical. We will then show that none of $\mathcal{K}, M_2, M_3, M_4$ or $M_7$ are $\UT$-tropical. This will complete the proof of the forward direction of Theorem~\ref{Thm:Main_UT}.

\
\begin{center}
\textit{Monogenic finite $\UT$-tropical matrices are aperiodic}
\end{center}

\begin{lemma}\label{Lem:C-k-ell-not-UT-tropical}
Let $C_{k,\ell} = \pres{Mon}{a}{a^k = a^\ell}$ with $0 \leq \ell < k$. If $C_{k,\ell}$ is not aperiodic, i.e. if $k>\ell+1$, then $C_{k,\ell}$ is not $\UT$-tropical. 
\end{lemma}
\begin{proof}
If $k > \ell +1$, then $\langle a^\ell \rangle$ is a cyclic group of order $k-\ell \geq 2$. However, the cyclic group $C_p$ does not embed in $\UTZ{n}$ when $p \geq 2$. Indeed, any matrix $A$ generating such a group $C_p$ would have to have all its diagonal entries as $0$ or $-\infty$ (as $(A^2)_{i,i} = A_{i,i} + A_{i,i}$), and a routine argument shows that any fixed off-diagonal entry for the matrices $A, A^2, A^3, \dots$ is eventually strictly increasing or else is eventually $-\infty$. If any such entry is strictly increasing, then $\langle A \rangle$ cannot be finite, and if all such entries are $-\infty$ for, say, $A^m$ for some sufficiently large $m$, then $A^{m+1} = A^m$, and hence $m = p = 1$, as desired.
\end{proof}

The author wishes to thank T. Aird for showing him the above proof.

\
\begin{center}
\textit{$M_2, M_3, M_4$, and $M_7$ are not $\UT$-tropical}
\end{center}

\begin{lemma}\label{Lem:M2-M3-not-UT}
Neither of the following two monoids are $\UT$-tropical:
\[
M_2 = \pres{Mon}{a,b}{a^2 = b^2} \quad \textnormal{and} \quad M_3 = \pres{Mon}{a,b}{aba=b}.
\]
\end{lemma}
\begin{proof}
As $M_2$ has no left or right cycles, it is cancellative \cite{Adian1966}. Consequently, by Corollary~\ref{Cor:Canc.-subsems-of-UTZ-are-comm}, it suffices to show that $M_2$ is not commutative. However, this is obvious -- no relation word is a subword of $ab$, and hence $ab \neq ba$ in $M_2$. Similarly, for $M_3$ it suffices to prove that $M_3$ is not commutative. This is easily proved, as the word $ab$ is clearly only equal to itself and words of the form $a^iba$ for $i>0$.
\end{proof}

We now show that $M_4$ and $M_7$ both contain a copy of $M_3$. This shows that neither of these monoids can be $\UT$-tropical, either.

\begin{lemma}\label{Lem:M3-in-M4-and-M7}
There exist (monoid) embeddings $M_3 \leq M_4$ and $M_3 \leq M_7$. 
\end{lemma}
\begin{proof}
First, in $M_4$, introduce a new generator $c = ba$. Then 
\begin{equation}\label{Eq:M_4new}
M_4 \cong \pres{Mon}{a,b,c}{aca = c, ba = c}.
\end{equation}
It is now clear that $\langle a, c \rangle_{M_4}$ has the presentation $M' = \pres{Mon}{a,c}{aca=c}$, i.e. it is isomorphic to $M_3$. Indeed, consider any sequence $s$ of elementary transformations from a word over $\{a, c \}$ into another. A routine argument now shows that if $s$ contains any elementary transformation $c \to ba$, then $s$ contains a turn (in the sense of Adian \cite{Adian1976}), and so this can be eliminated. Hence, by induction on the number of such transformations in $s$, any pair of words both over $\{ a, c\}$ and equal in $M_4$ are also equal in $M_3$, proving the claim. Thus $M_3 \leq M_4$. The case $M_3 \leq M_7$ is proved entirely analogously to $M_3 \leq M_4$, taking instead $c=ab$. 
\end{proof}

This completes the proof of the ``negative'' results of Theorem~\ref{Thm:Main_UT}, and hence also completes the proof of Theorem~\ref{Thm:Main_UT}. \qed

\section{One-relation monoids in $\MTZ{n}$}\label{Sec:Ordinary tropical}

\noindent In this section, we will describe the $\MT$-tropical one-relation monoids, i.e. those which embed into $\MTZ{n}$ for some $n \geq 1$. 

\begin{theoremB}\label{Thm:Main_MT}
Let $M$ be a one-relation monoid which is not $M_4$ or $M_7$. Then $M$ is $\MT$-tropical if and only if $M$ is monogenic, $\mathcal{K}$, the bicyclic monoid, or one of $M_2, M_3, M_5, M_6^{(k)}, M_8$, or $M_9^{(k)}$, for some $k \geq 1$.
\end{theoremB}

The cases $M_4$ and $M_7$, which are cases which have evaded the author's attempts at classifying, are treated below in \S\ref{Subsec:Open_m4_m7}. Note that we have already seen in \S\ref{Sec:Ordinary tropical} that the bicyclic monoid, $M_5, M_6^{(k)}, M_8$, and $M_9^{(k)}$ ($k \geq 1$) are all $\UT$-tropical, and hence also $\MT$-tropical. As the full transformation monoid $T_n$ is $\MT$-tropical (embedding into $\MTZ{n}$ via the analogue of permutation matrices), it follows that any finite semigroup is $\MT$-tropical, see also \cite[Theorem~3.3.1]{Taylor2017}. To show Theorem~\ref{Thm:Main_MT}, it hence suffices to show that $\mathcal{K}$, $M_2$, and $M_3$ are $\operatorname{M}$-tropical. 

\
\begin{center}
\textit{$M_2$ and $M_3$ are $\MT$-tropical}
\end{center}

We begin by showing that $M_3$ is $\operatorname{M}$-tropical, using a normal form lemma.

\begin{lemma}[Shneerson]\label{Lem:aba=b_normal_form}
Let $M_3 = \pres{Mon}{a,b}{aba=b}$. Then every word is equal in $M_3$ to some unique word either of the form $b^\alpha a^\beta$ for $\alpha, \beta \geq 0$, or of the form $b^\alpha a^\beta b$ for $\alpha \geq 0$, $\beta\geq 1$, and multiplication in $M_3$ is given by the rule 
\[
b^{\alpha_1} a^{\beta_1} b^{\alpha_2} a^{\beta_2} = \begin{cases}
b^{\alpha_1 + \alpha_2} a^{\beta_1 + \beta_2}, & \text{if $\alpha_2$ is even} \\
b^{\alpha_1 + \alpha_2} a^{\beta_1 - \beta_2}, & \text{if $\alpha_2$ is odd and $\beta_2 \geq \beta_1$} \\
b^{\alpha_1 + \alpha_2} a^{\beta_2 - \beta_1}, & \text{if $\alpha_2$ is odd and $\beta_1 \geq \beta_2$.} \\
\end{cases}
\]
\end{lemma}
The uniqueness of the representation is clear by the cancellativity of $M_3$. 

\begin{lemma}\label{Lem:aba=b-is_tropical_real}
The monoid $M_3 = \pres{Mon}{a,b}{aba=b}$ is isomorphic to the submonoid of $\MTZ{2}$ generated by the two matrices
\[
A = \begin{pmatrix}
1 & -\infty  \\
-\infty & -1 \\
\end{pmatrix} \quad \textnormal{and} \quad B = \begin{pmatrix}
-\infty & 1 \\
0 & -\infty \\
\end{pmatrix}.
\]
\end{lemma}
\begin{proof}
One checks directly that for all $\alpha, \beta \geq 0$, we have
\[
A^\beta = \begin{pmatrix}
\beta & -\infty  \\
-\infty & -\beta \\
\end{pmatrix}
\quad \text{and} \quad B^\alpha = \begin{cases} \begin{pmatrix}
\frac{\alpha}{2} & -\infty  \\
-\infty & \frac{\alpha}{2} \\
\end{pmatrix}, & \text{if $\alpha$ is even} \\
\begin{pmatrix}
-\infty & \frac{\alpha+1}{2}  \\
\frac{\alpha-1}{2} & -\infty \\
\end{pmatrix}, & \text{if $\alpha$ is odd}.
\end{cases}
\]
Hence
\begin{equation}\label{Eq:BalphaAbeta}
B^\alpha A^\beta = \begin{cases} \begin{pmatrix}
\frac{\alpha}{2}+\beta & -\infty  \\
-\infty & \frac{\alpha}{2}-\beta \\
\end{pmatrix}, & \text{if $\alpha$ is even} \\
\begin{pmatrix}
-\infty & \frac{\alpha+1}{2}-\beta  \\
\frac{\alpha-1}{2}+\beta & -\infty \\
\end{pmatrix}, & \text{if $\alpha$ is odd}.
\end{cases}
\end{equation}
\begin{equation}\label{Eq:BalphaAbetaBBBB}
B^\alpha A^\beta B = \begin{cases} \begin{pmatrix}
-\infty & \frac{\alpha}{2}+\beta+1   \\
\frac{\alpha}{2}-\beta & -\infty \\
\end{pmatrix}, & \text{if $\alpha$ is even} \\
\begin{pmatrix}
\frac{\alpha+1}{2}-\beta & -\infty  \\
-\infty & \frac{\alpha+1}{2}+\beta \\
\end{pmatrix}, & \text{if $\alpha$ is odd}.
\end{cases}
\end{equation}
It is clear that given any matrix of the form $B^\alpha A^\beta$ for $\alpha, \beta \geq 0$ or $B^\alpha A^\beta B$ for $\alpha \geq 0$, $\beta \geq 1$, we can determine $\alpha$ and $\beta$ uniquely, as well as determine whether we are in the first or the latter case. We give an example demonstrating this. Let $X = B^\alpha A^\beta B^\gamma$ for some arbitrarily chosen $(\alpha, \beta, \gamma) \in \mathbb{N} \times \mathbb{N} \times \{ 0 , 1\}$, such that $\beta=0$ only if $\gamma=0$, say 
\[
X = \begin{pmatrix}
-4 & -\infty  \\
-\infty & 12 \\
\end{pmatrix}.
\]
To determine which triple $(\alpha, \beta, \gamma)$ the author chose to produce this matrix $X$, we first note that by \eqref{Eq:BalphaAbeta} and \eqref{Eq:BalphaAbetaBBBB} we must, by observing the locations of the $-\infty$-symbols, have that either $\alpha$ is even and $\gamma=0$, or else $\alpha$ is odd and $\gamma=1$. In the case $\gamma=0$, we would have $X_{1,1} > X_{2,2}$; in the case $\gamma=1$, we would have $X_{1,1} < X_{2,2}$ (the case $X_{1,1}=X_{2,2}$ happens if and only if $\beta=0$, which is incompatible with $\gamma=1$). As $-4 < 12$, we have $\gamma=1$, so $\alpha$ is odd. We hence have the system 
\begin{equation*}
\begin{cases} \frac{\alpha+1}{2} -\beta =  -4 \\ \frac{\alpha+1}{2} +\beta = 12 
\end{cases}
\end{equation*}
which has the unique solution $\alpha = 7, \beta = 8$. Thus $X$ uniquely determines the triple $(\alpha, \beta, \gamma) = (7, 8, 1)$, and one can check directly that $B^7 A^8 B = X$. We conclude, by similar reasoning in the other cases, that $\psi$ is injective, completing the proof. 
\end{proof}

\begin{lemma}\label{Lem:M2-in-M3}
There exists a (monoid) embedding $M_2 \leq M_3$.  
\end{lemma}
\begin{proof}
Let $G_2 = \pres{Gp}{a,b}{a^2 = b^2}$ and $G_3 = \pres{Gp}{a,b}{aba = b}$. Then $M_2 \leq G_2$ and $M_3 \leq G_3$ via the identity maps, as both $M_2, M_3$ are cycle-free. Note that $G_2 \cong G_3$ (though $M_2 \not\cong M_3$) via the isomorphism $\psi$ induced by the free group automorphism $a \mapsto ba$ and $b \mapsto b$. Indeed, $\psi(b^{-2}a^2) = b^{-2}baba = b^{-1}ab$. Now $\psi$ restricts to give an isomorphism $M_2 \cong \langle \psi(a), \psi(b) \rangle_{M_3} = \langle ba, b \rangle_{M_3}$, as desired. 
\end{proof}

Now, from Lemma~\ref{Lem:aba=b-is_tropical_real} and Lemma~\ref{Lem:M2-in-M3} (and the hereditary property of $\MT$-tropicality), we conclude that $M_2$ is also $\operatorname{M}$-tropical; indeed, we find the following explicit embedding by using the fact that $\langle ba,b \rangle_{M_3} \cong M_2$.

\begin{lemma}\label{Lem:a2=b2-is_tropical_real}
The monoid $M_2 = \pres{Mon}{a,b}{a^2=b^2}$ is isomorphic to the submonoid of $\MTZ{2}$ generated by the two matrices
\[
A = \begin{pmatrix}
-\infty & 0 \\
1 & -\infty \\
\end{pmatrix} \qquad \textnormal{and} \qquad B = \begin{pmatrix}
-\infty & 1 \\
0 & -\infty \\
\end{pmatrix}.
\]
\end{lemma}

This completes the proof that $M_2$ and $M_3$ are $\MT$-tropical.

\
\begin{center}
\textit{The Klein bottle group $\mathcal{K}$ is $\MT$-tropical}
\end{center}

Shitov \cite{Shitov2012} has proved that every subgroup of a $\MTZ{n}$ is virtually abelian. We will use a converse to this, due to Shitov, to prove that the virtually abelian group $\mathcal{K}$ embeds in $\MTZ{2}$, although extracting two matrices generating $\mathcal{K}$ from the method seems somewhat intricate. Hence, we also give a direct proof via an explicit generating set. 

\begin{lemma}\label{Lem:Klein-bottle-embeds-in-U}
The fundamental group of the Klein bottle $\mathcal{K}= \pres{Mon}{a,b}{abba=1}$ is isomorphic to the subsemigroup of $\MTZ{2}$ generated by the two matrices
\[
A = \begin{pmatrix}
-\infty & -1 \\
0 & -\infty 
\end{pmatrix} \qquad \textnormal{and} \qquad B = 
\begin{pmatrix}
-\infty & 1 \\
0 & -\infty 
\end{pmatrix}.
\]
\end{lemma}
\begin{proof}
It follows from Shitov \cite{Shitov2012} that a group $G$ admits a faithful representation in $\MTZ{n}$ if and only if $G$ is isomorphic to a subgroup of $\mathbb{Z} \wr S_n$, where $S_n$ is the symmetric group. In the group 
\[
\mathbb{Z} \wr S_2 \cong \pres{Gp}{a,b}{a^2 = [a, bab^{-1}] = 1}
\]
it is not difficult to see, by using the Reidemeister--Schreier method, that the subgroup $H = \langle ab, ab^{-1}\rangle \leq \mathbb{Z} \wr S_2$ is a normal subgroup of index $2$ and that $H \cong \mathcal{K}$. This proves the claim. 

To prove that  $A$ and $B$ generate $\mathcal{K}$, first note that $ABBA = 1$. Jantzen \cite{Jantzen1985} gave the following normal form: every element of $\mathcal{K}$ can be written as $w = b^i(ab)^ja^k$ for some $i, j, k \in \mathbb{N}$, and the triple $(i,j,k)$ is uniquely determined by $w$ modulo replacing $i, k$ by $i+2p, k+2p$ (for some $p \in \mathbb{Z}$), see \cite[Lemma~6 \& Corollary~7]{Jantzen1985}. A lengthy, but straightforward, calculation (herein omitted) shows that any product $B^i(AB)^jA^k$ is unique up to the same condition, proving the result. 
\end{proof}

This completes the proof of Theorem~\ref{Thm:Main_MT}. \qed

\subsection{The open cases ($M_4$ and $M_7$)}\label{Subsec:Open_m4_m7}

We now turn briefly towards the two one-relation monoids which have, as of yet, eluded the author's attempts prove whether or not they are $\operatorname{M}$-tropical. First, note that as $M_4$ and $M_7$ are anti-isomorphic, one is $\MT$-tropical if and only if the other is. We certainly believe that neither is $\MT$-tropical, but the author's attempted proofs of this have quickly run into the technicalities of tropical mathematics. We therefore pose the following conjecture:

\begin{conjecture}\label{Conj:ABAA=BA-is-not-tropical}
The monoid $M_4 = \pres{Mon}{a,b}{aba^2 = ba}$ does not admit a faithful tropical representation into $\MTZ{n}$ for any $n \geq 1$ (hence, neither does $M_7$). 
\end{conjecture}

By testing a large number of tropical matrices, it seems that the class of matrices $A, B \in \MTZ{n}$ satisfying $A(BA)A = BA$ is very restrictive, especially if one further assumes easy conditions derived from the structure of $M_4$ (e.g. $A, B$ both have infinite order as $M_4$ has no non-trivial idempotents, $AB \neq BA$ as $ab \neq ba$ in $M_4$, etc.). Note further that $A^k (BA) A^k = A^i (BA) A^i$ for all $0 \leq i \leq k$, which seems to be a lot to demand of matrices $A$ and $B$ both of infinite order. This style of reasoning, though not formalised, is what leads us to Conjecture~\ref{Conj:ABAA=BA-is-not-tropical}. Either way, if proved or disproved, a resolution of Conjecture~\ref{Conj:ABAA=BA-is-not-tropical} would complete the classification of the $\operatorname{M}$-tropical matrices.

\section*{Acknowledgements}

\noindent I wish to express my sincere thanks to L. M. Shneerson (Hunter College, CUNY), first for his encouragement and interest shown in my work, and second for all his remarkable work in combinatorial semigroup theory, particularly \cite{Shneerson1972a, Shneerson1972b}, copies of which he kindly sent to me, and without which this article could not have been written. I also wish to thank T. Aird (Manchester) for sharing his expertise on tropical mathematics on numerous occasions, and M. Johnson \& M. Kambites (Manchester) for helpful discussions and for introducing me to the tropical world.

{
\bibliography{tropicalORM.bib} 
\bibliographystyle{amsalpha}
}
\end{document}